%% file: vectorflows.tex
\documentclass[11pt]{amsart} 
\usepackage{amssymb}
\usepackage{amsmath}
\usepackage{setspace}
\usepackage{amsthm}
\usepackage{epsfig,color,colordvi,wasysym}
\usepackage{tipa}
\usepackage{tikz}
\usepackage{verbatim}

\newtheorem{theorem}{Theorem}

\def\Z{\mathbb{Z}}

\def\R{\mathbb{R}}

\def\H{\mathcal{H}}
\def\P{\mathcal{P}}

\def\k{\mathbf{k}}

\def\x{\mathbf{x}}

\def\0{\mathbf{0}}
\def\1{\mathbf{1}}

\begin{document}

\title{Nowhere-Zero $\vec k$-Flows on Graphs}

\author[M.~Beck]{Matthias Beck}
\address{Department of Mathematics, San Francisco State University, San Francisco, CA 94132, USA}
\email{mattbeck@sfsu.edu}

\author[A.~Cuyjet]{Alyssa Cuyjet}
\address{Department of Mathematics, Trinity College, Hartford, CT 06106, USA}
\email{alyssa.cuyjet@trincoll.edu}

\author[G.~R.~Kirby]{Gordon Rojas Kirby}
\address{Department of Mathematics, Stanford University, Palo Alto, CA 94305, USA}
\email{girkirby@gmail.com}

\author[M.~Stubblefield]{Molly Stubblefield}
\address{Mathematics Department, Western Oregon University, Monmouth, OR 97361, USA}
\email{mstubblefield08@mail.wou.edu}

\author[M.~Young]{Michael Young}
\address{Department of Mathematics, Iowa State University, Ames, IA 50011, USA}
\email{myoung@iastate.edu}

\begin{abstract}
We introduce and study a multivariate function that counts nowhere-zero flows on a graph $G$, in which each edge of $G$ has an
individual capacity. We prove that the associated counting function is a piecewise-defined polynomial in these capacities, which
satisfy a combinatorial reciprocity law that incorporates totally cyclic orientations of~$G$. 
\end{abstract}

\keywords{Graph flow, nowhere-zero flow, multivariate piecewise-defined polynomial, combinatorial reciprocity theorem, totally cyclic orientation}

\subjclass[2000]{Primary 05C21; Secondary 05A15, 05C31.}

\date{10 May 2013}

\thanks{We thank Ricardo Cortez and the staff at MSRI for creating an ideal research environment at MSRI-UP, and
Felix Breuer for stimulating discussions about our work.
This research was partially supported by the NSF through the grants DMS-1162638 (Beck),
DMS-0946431 (Young), and DMS-1156499 (MSRI-UP REU), and by the NSA through grant H98230-11-1-0213.}

\maketitle



Let $G=(V,E)$ be a graph, which we allow to have multiple edges.
Fix an orientation of $G$, i.e., for each edge $e = uv \in E$ we assign one of $u$ and $v$ to be the head $h(e)$ and the other to be the tail
$t(e)$ of $e$.
A \emph{flow} on (this orientation of) $G$ is a labeling $\x \in A^E$ of the edges of $G$ with values in an Abelian group $A$ such that for every node $v \in V$,
\[
  \sum_{ h(e) = v } x_e \ = \ \sum_{ t(e) = v } x_e \, , 
\]
i.e., we have conservation of flow at each node.
We are interested in counting such flows that are \emph{nowhere zero}, i.e., $x_e \ne 0$ for every edge $e \in E$.

\def\modflow{\overline \varphi}
\def\flow{\varphi}

The case $A = \Z_k$ goes back to Tutte \cite{tutteflowpoly}, who proved that the number $\modflow (k)$ of nowhere-zero $\Z_k$-flows is a polynomial in $k$.
(Tutte introduced this counting function as a dual concept to the chromatic polynomial.)
In the case $A = \Z$, a \emph{$k$-flow} takes on integer values in $\{ -k+1, \dots, k-1 \}$. The fact that the number $\flow (k)$ of nowhere-zero $k$-flows is also a polynomial in $k$ is due to Kochol \cite{kocholpolynomial}.
Both $\modflow (k)$ and $\flow (k)$ are easily seen to be independent of the chosen orientation of $G$, and so we will write these flow
polynomials as $\modflow_G (k)$ and $\modflow_G (k)$.
We give three examples of each in \eqref{smallexamples} below.
We also note the well-known fact that $G$ admits no nowhere-zero flow if $G$ has a \emph{bridge} (also known as an \emph{isthmus}), i.e., an edge whose removal increases the number of components of~$G$.

Since both flow counting functions are polynomials, it is natural to ask about evaluations other than at positive integers.
Mirroring a famous result of Stanley \cite{stanleyacyclic} on chromatic polynomials and acyclic orientations of a graph, this question was answered for negative
integers (instances of a \emph{combinatorial reciprocity theorem}) by Beck--Zaslavsky \cite{nnz} for $\flow_G(k)$ and by Breuer--Sanyal \cite{breuersanyal} for
$\modflow_G(k)$; in both cases the counting function yields a connection to totally cyclic orientations of the graph.

Our goal is to introduce and study a flow counting function depending on several variables. 
We define a \emph{$\k$-flow} on $G$, for $\k \in \Z_{ >0 }^E$, as an integral flow $\x \in \Z^E$ such that $|x_e| < k_e$ for each $e \in E$. 
In plain English, we have a different capacity for each edge. While this seems to be a natural concept for graph flows, we are not aware of any enumeration concept
associated with it.

As before, we count only the nowhere-zero $\k$-flows of $G$, and we denote this count (with a slight abuse of notation) by $\flow_G(\k)$.
Below are three examples for this multivariate flow function. 

\def\JPicScale{.8}
 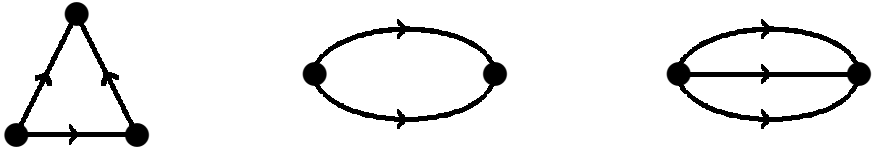
\begin{figure}[!htb]
\begin{center}
\input{smallGraphs}
\end{center}
\label{smallGraphs}
\vspace{-5 cm}
\caption{Three small graphs: $K_3$, $2K_2$, and $3K_2$.}
\end{figure}

\begin{align}
  \modflow_{ K_3 } (k) &= k-1 \nonumber \\
  \flow_{ K_3 } (k) &= 2(k-1) \nonumber \\
  \flow_{ K_3 } (k_1, k_2, k_3) &= 2(k_1-1) \ \text{ if } k_1 \le k_2 \le k_3 \nonumber \\
  \modflow_{ 2K_2 } (k) &= k-1 \nonumber \\
  \flow_{ 2K_2 } (k) &= 2(k-1) \nonumber \\
  \flow_{ 2K_2 } (k_1, k_2) &= 2(k_1-1) \ \text{ if } k_1 \le k_2 \label{smallexamples} \\
  \modflow_{ 3K_2 } (k) &= (k-1)(k-2) \nonumber \\
  \flow_{ 3K_2 } (k) &= 3(k-1)(k-2) \nonumber \\
  \flow_{ 3K_2 } (k_1, k_2, k_3) &= \begin{cases}
    (2k_1-2)(2k_2-3) & \text{ if } k_1 \leq k_2 \leq k_3 > k_2+k_1 \, , \\
    -k_1^2+2k_1k_2+2k_1k_3-5k_1 -k_2^2 \\
    \quad \ {} +2k_2k_3-3k_2-k_3^2-k_3+6 & \text{ if } k_1 \leq k_2 \leq k_3 \leq k_2+k_1 \, .
  \end{cases} \nonumber
\end{align}
Already in these small examples, we can see some similarities (e.g., the degrees of the flow polynomials associated with a given graph, or the constant terms of
$\flow_G(k)$ and $\flow_G(\k)$) and some differences among the three flow counting functions, most notably the fact that $\flow_G(\k)$ is only a \emph{piecewise-defined} polynomial in $\k$ (note that in \eqref{smallexamples} we state $\flow_G(\k)$ in each case only for $k_1 \le k_2 \le \cdots$; we may do so in these cases due to the symmetry of the graphs considered).
A moment's thought reveals that this is structurally the best we can hope for $\flow_G(\k)$, and in fact, our first result says this structure always holds.
For a graph $G = (V, E)$, let $\xi_G := |E| - |V| + \#\mathrm{components}(G)$, the \emph{cyclomatic number} of~$G$.

\begin{theorem}\label{polthm}
For a bridgeless graph $G$, the multivariate flow counting function $\flow_G(\k)$ is a piecewise-defined polynomial of degree~$\xi_G$.
\end{theorem}

Our second 
result is a combinatorial reciprocity theorem for $\flow_G(\k)$, which gives a vector-valued generalization of the main result in \cite{nnz}.
To state it, recall that a \emph{totally cyclic orientation} of $G$ is an orientation for which every edge lies in a coherently-oriented cycle.
A $\k$-flow of $G$ gives naturally rise to orientations of $G$ by changing the initial orientation of $G$ as
follows:
\begin{itemize}
  \item switch the direction of any edge that has a negative flow label;
  \item leave the direction of any edge that has a positive flow label;
  \item switch or leave the direction of any edge that has a zero flow label.
\end{itemize}
Any such orientation is \emph{compatible} with the $\k$-flow.
Denote by $\1 \in \Z^E$ a vector all of whose entries are~$1$.

\begin{theorem}\label{recthm}
Let $G$ be a bridgeless graph. Then $(-1)^{ \xi_G } \flow_G(-\k)$ equals the number of $(\k+\1)$-flows of $G$, each counted with multiplicity equal to the number of compatible totally cyclic orientations of $G$. 
In particular, $(-1)^{ \xi_G } \flow_G(\0)$ equals the number of totally cyclic orientations of~$G$. 
\end{theorem}

\begin{proof}[Proof of Theorem \ref{polthm}]
Given a graph $G$ with a fixed orientation, let $A \in \Z^{ V \times E }$ be its \emph{(signed) incidence matrix} with entries
\[
  a_{ ve } = \begin{cases}
     1 & \text{ if } v = h(e) , \\
    -1 & \text{ if } v = t(e) , \\
     0 & \text{ otherwise. } 
  \end{cases} 
\]
Let $F$ be the kernel of $A$, viewed as a subspace of $\R^E$. Note that, by definition, a nowhere-zero integral flow of $G$ is an integer lattice point in
\[
  F \setminus \H
  \qquad \text{ where } \qquad
  \H := \left\{ x_e = 0 : \, e \in E \right\} , 
\]
the arrangement of coordinate hyperplanes in $\R^E$. The induced arrangement in $F$ has regions (i.e., maximal connected components of $F \setminus \H$) which are in
one-to-one correspondence with the totally cyclic orientations of $G$ \cite{greenezaslavsky}. 
Taking a leaf from the geometric setup of \cite{iop} (which was also used in \cite{nnz}), given $\k \in \Z_{ >0 }^E$, a $\k$-flow is precisely a point in
\[
  \left( [-k_1, k_1] \times \dots \times [-k_{ |E| } , k_{ |E| }] \right)^\circ \cap \Z^E \cap F \setminus \H \, . 
\]
(Here $\P^\circ$ denotes the interior of $\P$.)
This geometric object is a union of open polytopes $\P_\sigma^\circ$ (depending on $\k$) which are naturally indexed by the totally cyclic orientations of $G$.
More precisely, if we let
\[
  \P_\sigma(\k) := \left( [0, k_1] \times \dots \times [0, k_{ |E| }] \right) \cap F_\sigma \, ,
\]
where $F_\sigma$ denotes the kernel of the incidence matrix $A_\sigma$ of $G$ reoriented by $\sigma$,
\[
  f_\sigma(\k) := \# \left( \P_\sigma(\k)^\circ \cap \Z^E \right) ,
\]
and $\Omega$ the set of all the totally cyclic orientations of $G$, then
\begin{equation}\label{flowassum}
  \flow_G(\k) = \sum_{ \sigma \in \Omega } f_\sigma(\k) .
\end{equation}
Thus it suffices to prove that $\# \left( \P_\sigma(\k)^\circ \cap \Z^E \right)$ is a piecewise-defined polynomial of degree $\xi_G$.
The structure of this counting function is that of a \emph{vector partition function}:
\[
  \# \left( \P_\sigma(\k)^\circ \cap \Z^E \right) =
  \# \left\{ \x \in \Z^E : \, A_\sigma \, \x = 0 \, , \ 0 < x_e < k_e \right\} .
\]
It is well known (see, e.g., \cite{schrijvertheorylinintprog}) that $A_\sigma$ is \emph{totally unimodular}, i.e., every square submatrix of $A_\sigma$ has
determinant $\pm 1$ or 0. But from this we can conclude with the general theory of vector partition functions (see, e.g.,
\cite{dahmenmicchelli,sturmfelsvectorpartition}) that $\# \left( \P_\sigma(\k)^\circ \cap \Z^E \right)$ is a piecewise-defined polynomial. 
Its degree equals $\xi_G$ because that is the dimension of the underlying polytope~\cite{nnz}.
\end{proof}

\begin{proof}[Proof of Theorem \ref{recthm}]
We start with \eqref{flowassum} and use the main result of~\cite{evencloser}:
\[
  \flow_G(-\k) 
  = \sum_{ \sigma \in \Omega } f_\sigma(-\k)
  = \sum_{ \sigma \in \Omega } (-1)^{ \xi_G } \, \# \left( \P_\sigma(\k) \cap \Z^E \right) . 
\]
What we are counting on the right-hand side (ignoring $(-1)^{ \xi_G }$) are $(\k+\1)$-flows with multiplicities that come from zero entries, i.e., a
zero label of an edge.
This multiplicity is precisely the number of closed regions of the hyperplane arrangement induced by $\H$ on $F$ (using the notation from the beginning of our proof
of Theorem \ref{polthm}) a $(\k+\1)$-flow lies in (viewed as a point in $F$). 
The afore-mentioned interpretation of these regions in terms of totally cyclic orientations of $G$ \cite{greenezaslavsky} gives Theorem~\ref{recthm}. 
\end{proof}

We close with two open problems.
While the $\Z_k$-flow polynomial $\modflow_G (k)$ is always monic, the problem of finding a formula for, or a combinatorial interpretation of, the leading coefficient of
the integer flow polynomial $\flow_G(k)$ remains open \cite[Problem 3.2]{nnz}; by the underlying geometry, this leading coefficient equals the leading coefficient of~$\flow_G(\k)$.

The algebraic structure of the new flow counting function $\flow_G(\k)$ gives a natural new problem, namely,
to determine its regions polynomiality.
The analogous problem for general vector partition functions (assuming the underlying matrix is totally unimodular;
if not, there is a variant of this problem) is wide open.
As we have seen, $\flow_G(\k)$ is a sum of a special case of vector partition functions. The problem of determining
the regions of polynomiality for $\flow_G(\k)$ might be as intractable as that for general vector partition
functions, but one can hope that the special form of multivariate flow polynomials allows some insights about regions of polynomiality.


\bibliographystyle{amsplain}

\def\cprime{$'$} \def\cprime{$'$}
\providecommand{\bysame}{\leavevmode\hbox to3em{\hrulefill}\thinspace}
\providecommand{\MR}{\relax\ifhmode\unskip\space\fi MR }
\providecommand{\MRhref}[2]{%
  \href{http://www.ams.org/mathscinet-getitem?mr=#1}{#2}
}
\providecommand{\href}[2]{#2}

\end{document}

%% file: smallGraphs
\ifx\JPicScale\undefined\def\JPicScale{1}\fi
\unitlength \JPicScale mm
\begin{picture}(142.03,92.03)(0,0)
\linethickness{0.55mm}
\put(140,79.87){\line(0,1){0.25}}
\multiput(139.98,79.62)(0.02,0.25){1}{\line(0,1){0.25}}
\multiput(139.95,79.37)(0.03,0.25){1}{\line(0,1){0.25}}
\multiput(139.9,79.12)(0.05,0.25){1}{\line(0,1){0.25}}
\multiput(139.83,78.88)(0.07,0.25){1}{\line(0,1){0.25}}
\multiput(139.75,78.63)(0.08,0.25){1}{\line(0,1){0.25}}
\multiput(139.65,78.38)(0.1,0.25){1}{\line(0,1){0.25}}
\multiput(139.53,78.14)(0.12,0.24){1}{\line(0,1){0.24}}
\multiput(139.4,77.9)(0.13,0.24){1}{\line(0,1){0.24}}
\multiput(139.25,77.66)(0.15,0.24){1}{\line(0,1){0.24}}
\multiput(139.09,77.42)(0.16,0.24){1}{\line(0,1){0.24}}
\multiput(138.91,77.19)(0.09,0.12){2}{\line(0,1){0.12}}
\multiput(138.71,76.96)(0.1,0.12){2}{\line(0,1){0.12}}
\multiput(138.5,76.73)(0.11,0.11){2}{\line(0,1){0.11}}
\multiput(138.27,76.51)(0.11,0.11){2}{\line(1,0){0.11}}
\multiput(138.03,76.29)(0.12,0.11){2}{\line(1,0){0.12}}
\multiput(137.78,76.07)(0.13,0.11){2}{\line(1,0){0.13}}
\multiput(137.51,75.86)(0.13,0.11){2}{\line(1,0){0.13}}
\multiput(137.22,75.65)(0.14,0.1){2}{\line(1,0){0.14}}
\multiput(136.93,75.45)(0.15,0.1){2}{\line(1,0){0.15}}
\multiput(136.62,75.25)(0.16,0.1){2}{\line(1,0){0.16}}
\multiput(136.29,75.06)(0.16,0.1){2}{\line(1,0){0.16}}
\multiput(135.96,74.88)(0.17,0.09){2}{\line(1,0){0.17}}
\multiput(135.61,74.7)(0.17,0.09){2}{\line(1,0){0.17}}
\multiput(135.25,74.52)(0.36,0.17){1}{\line(1,0){0.36}}
\multiput(134.87,74.35)(0.37,0.17){1}{\line(1,0){0.37}}
\multiput(134.49,74.19)(0.38,0.16){1}{\line(1,0){0.38}}
\multiput(134.1,74.04)(0.39,0.16){1}{\line(1,0){0.39}}
\multiput(133.69,73.89)(0.4,0.15){1}{\line(1,0){0.4}}
\multiput(133.28,73.75)(0.41,0.14){1}{\line(1,0){0.41}}
\multiput(132.86,73.61)(0.42,0.13){1}{\line(1,0){0.42}}
\multiput(132.43,73.48)(0.43,0.13){1}{\line(1,0){0.43}}
\multiput(131.99,73.36)(0.44,0.12){1}{\line(1,0){0.44}}
\multiput(131.54,73.25)(0.45,0.11){1}{\line(1,0){0.45}}
\multiput(131.09,73.15)(0.45,0.11){1}{\line(1,0){0.45}}
\multiput(130.62,73.05)(0.46,0.1){1}{\line(1,0){0.46}}
\multiput(130.16,72.96)(0.47,0.09){1}{\line(1,0){0.47}}
\multiput(129.68,72.87)(0.47,0.08){1}{\line(1,0){0.47}}
\multiput(129.2,72.8)(0.48,0.07){1}{\line(1,0){0.48}}
\multiput(128.72,72.73)(0.48,0.07){1}{\line(1,0){0.48}}
\multiput(128.23,72.68)(0.49,0.06){1}{\line(1,0){0.49}}
\multiput(127.74,72.63)(0.49,0.05){1}{\line(1,0){0.49}}
\multiput(127.25,72.58)(0.49,0.04){1}{\line(1,0){0.49}}
\multiput(126.75,72.55)(0.5,0.03){1}{\line(1,0){0.5}}
\multiput(126.25,72.53)(0.5,0.03){1}{\line(1,0){0.5}}
\multiput(125.75,72.51)(0.5,0.02){1}{\line(1,0){0.5}}
\multiput(125.25,72.5)(0.5,0.01){1}{\line(1,0){0.5}}
\put(124.75,72.5){\line(1,0){0.5}}
\multiput(124.25,72.51)(0.5,-0.01){1}{\line(1,0){0.5}}
\multiput(123.75,72.53)(0.5,-0.02){1}{\line(1,0){0.5}}
\multiput(123.25,72.55)(0.5,-0.03){1}{\line(1,0){0.5}}
\multiput(122.75,72.58)(0.5,-0.03){1}{\line(1,0){0.5}}
\multiput(122.26,72.63)(0.49,-0.04){1}{\line(1,0){0.49}}
\multiput(121.77,72.68)(0.49,-0.05){1}{\line(1,0){0.49}}
\multiput(121.28,72.73)(0.49,-0.06){1}{\line(1,0){0.49}}
\multiput(120.8,72.8)(0.48,-0.07){1}{\line(1,0){0.48}}
\multiput(120.32,72.87)(0.48,-0.07){1}{\line(1,0){0.48}}
\multiput(119.84,72.96)(0.47,-0.08){1}{\line(1,0){0.47}}
\multiput(119.38,73.05)(0.47,-0.09){1}{\line(1,0){0.47}}
\multiput(118.91,73.15)(0.46,-0.1){1}{\line(1,0){0.46}}
\multiput(118.46,73.25)(0.45,-0.11){1}{\line(1,0){0.45}}
\multiput(118.01,73.36)(0.45,-0.11){1}{\line(1,0){0.45}}
\multiput(117.57,73.48)(0.44,-0.12){1}{\line(1,0){0.44}}
\multiput(117.14,73.61)(0.43,-0.13){1}{\line(1,0){0.43}}
\multiput(116.72,73.75)(0.42,-0.13){1}{\line(1,0){0.42}}
\multiput(116.31,73.89)(0.41,-0.14){1}{\line(1,0){0.41}}
\multiput(115.9,74.04)(0.4,-0.15){1}{\line(1,0){0.4}}
\multiput(115.51,74.19)(0.39,-0.16){1}{\line(1,0){0.39}}
\multiput(115.13,74.35)(0.38,-0.16){1}{\line(1,0){0.38}}
\multiput(114.75,74.52)(0.37,-0.17){1}{\line(1,0){0.37}}
\multiput(114.39,74.7)(0.36,-0.17){1}{\line(1,0){0.36}}
\multiput(114.04,74.88)(0.17,-0.09){2}{\line(1,0){0.17}}
\multiput(113.71,75.06)(0.17,-0.09){2}{\line(1,0){0.17}}
\multiput(113.38,75.25)(0.16,-0.1){2}{\line(1,0){0.16}}
\multiput(113.07,75.45)(0.16,-0.1){2}{\line(1,0){0.16}}
\multiput(112.78,75.65)(0.15,-0.1){2}{\line(1,0){0.15}}
\multiput(112.49,75.86)(0.14,-0.1){2}{\line(1,0){0.14}}
\multiput(112.22,76.07)(0.13,-0.11){2}{\line(1,0){0.13}}
\multiput(111.97,76.29)(0.13,-0.11){2}{\line(1,0){0.13}}
\multiput(111.73,76.51)(0.12,-0.11){2}{\line(1,0){0.12}}
\multiput(111.5,76.73)(0.11,-0.11){2}{\line(1,0){0.11}}
\multiput(111.29,76.96)(0.11,-0.11){2}{\line(0,-1){0.11}}
\multiput(111.09,77.19)(0.1,-0.12){2}{\line(0,-1){0.12}}
\multiput(110.91,77.42)(0.09,-0.12){2}{\line(0,-1){0.12}}
\multiput(110.75,77.66)(0.16,-0.24){1}{\line(0,-1){0.24}}
\multiput(110.6,77.9)(0.15,-0.24){1}{\line(0,-1){0.24}}
\multiput(110.47,78.14)(0.13,-0.24){1}{\line(0,-1){0.24}}
\multiput(110.35,78.38)(0.12,-0.24){1}{\line(0,-1){0.24}}
\multiput(110.25,78.63)(0.1,-0.25){1}{\line(0,-1){0.25}}
\multiput(110.17,78.88)(0.08,-0.25){1}{\line(0,-1){0.25}}
\multiput(110.1,79.12)(0.07,-0.25){1}{\line(0,-1){0.25}}
\multiput(110.05,79.37)(0.05,-0.25){1}{\line(0,-1){0.25}}
\multiput(110.02,79.62)(0.03,-0.25){1}{\line(0,-1){0.25}}
\multiput(110,79.87)(0.02,-0.25){1}{\line(0,-1){0.25}}
\put(110,79.87){\line(0,1){0.25}}
\multiput(110,80.13)(0.02,0.25){1}{\line(0,1){0.25}}
\multiput(110.02,80.38)(0.03,0.25){1}{\line(0,1){0.25}}
\multiput(110.05,80.63)(0.05,0.25){1}{\line(0,1){0.25}}
\multiput(110.1,80.88)(0.07,0.25){1}{\line(0,1){0.25}}
\multiput(110.17,81.12)(0.08,0.25){1}{\line(0,1){0.25}}
\multiput(110.25,81.37)(0.1,0.25){1}{\line(0,1){0.25}}
\multiput(110.35,81.62)(0.12,0.24){1}{\line(0,1){0.24}}
\multiput(110.47,81.86)(0.13,0.24){1}{\line(0,1){0.24}}
\multiput(110.6,82.1)(0.15,0.24){1}{\line(0,1){0.24}}
\multiput(110.75,82.34)(0.16,0.24){1}{\line(0,1){0.24}}
\multiput(110.91,82.58)(0.09,0.12){2}{\line(0,1){0.12}}
\multiput(111.09,82.81)(0.1,0.12){2}{\line(0,1){0.12}}
\multiput(111.29,83.04)(0.11,0.11){2}{\line(0,1){0.11}}
\multiput(111.5,83.27)(0.11,0.11){2}{\line(1,0){0.11}}
\multiput(111.73,83.49)(0.12,0.11){2}{\line(1,0){0.12}}
\multiput(111.97,83.71)(0.13,0.11){2}{\line(1,0){0.13}}
\multiput(112.22,83.93)(0.13,0.11){2}{\line(1,0){0.13}}
\multiput(112.49,84.14)(0.14,0.1){2}{\line(1,0){0.14}}
\multiput(112.78,84.35)(0.15,0.1){2}{\line(1,0){0.15}}
\multiput(113.07,84.55)(0.16,0.1){2}{\line(1,0){0.16}}
\multiput(113.38,84.75)(0.16,0.1){2}{\line(1,0){0.16}}
\multiput(113.71,84.94)(0.17,0.09){2}{\line(1,0){0.17}}
\multiput(114.04,85.12)(0.17,0.09){2}{\line(1,0){0.17}}
\multiput(114.39,85.3)(0.36,0.17){1}{\line(1,0){0.36}}
\multiput(114.75,85.48)(0.37,0.17){1}{\line(1,0){0.37}}
\multiput(115.13,85.65)(0.38,0.16){1}{\line(1,0){0.38}}
\multiput(115.51,85.81)(0.39,0.16){1}{\line(1,0){0.39}}
\multiput(115.9,85.96)(0.4,0.15){1}{\line(1,0){0.4}}
\multiput(116.31,86.11)(0.41,0.14){1}{\line(1,0){0.41}}
\multiput(116.72,86.25)(0.42,0.13){1}{\line(1,0){0.42}}
\multiput(117.14,86.39)(0.43,0.13){1}{\line(1,0){0.43}}
\multiput(117.57,86.52)(0.44,0.12){1}{\line(1,0){0.44}}
\multiput(118.01,86.64)(0.45,0.11){1}{\line(1,0){0.45}}
\multiput(118.46,86.75)(0.45,0.11){1}{\line(1,0){0.45}}
\multiput(118.91,86.85)(0.46,0.1){1}{\line(1,0){0.46}}
\multiput(119.38,86.95)(0.47,0.09){1}{\line(1,0){0.47}}
\multiput(119.84,87.04)(0.47,0.08){1}{\line(1,0){0.47}}
\multiput(120.32,87.13)(0.48,0.07){1}{\line(1,0){0.48}}
\multiput(120.8,87.2)(0.48,0.07){1}{\line(1,0){0.48}}
\multiput(121.28,87.27)(0.49,0.06){1}{\line(1,0){0.49}}
\multiput(121.77,87.32)(0.49,0.05){1}{\line(1,0){0.49}}
\multiput(122.26,87.37)(0.49,0.04){1}{\line(1,0){0.49}}
\multiput(122.75,87.42)(0.5,0.03){1}{\line(1,0){0.5}}
\multiput(123.25,87.45)(0.5,0.03){1}{\line(1,0){0.5}}
\multiput(123.75,87.47)(0.5,0.02){1}{\line(1,0){0.5}}
\multiput(124.25,87.49)(0.5,0.01){1}{\line(1,0){0.5}}
\put(124.75,87.5){\line(1,0){0.5}}
\multiput(125.25,87.5)(0.5,-0.01){1}{\line(1,0){0.5}}
\multiput(125.75,87.49)(0.5,-0.02){1}{\line(1,0){0.5}}
\multiput(126.25,87.47)(0.5,-0.03){1}{\line(1,0){0.5}}
\multiput(126.75,87.45)(0.5,-0.03){1}{\line(1,0){0.5}}
\multiput(127.25,87.42)(0.49,-0.04){1}{\line(1,0){0.49}}
\multiput(127.74,87.37)(0.49,-0.05){1}{\line(1,0){0.49}}
\multiput(128.23,87.32)(0.49,-0.06){1}{\line(1,0){0.49}}
\multiput(128.72,87.27)(0.48,-0.07){1}{\line(1,0){0.48}}
\multiput(129.2,87.2)(0.48,-0.07){1}{\line(1,0){0.48}}
\multiput(129.68,87.13)(0.47,-0.08){1}{\line(1,0){0.47}}
\multiput(130.16,87.04)(0.47,-0.09){1}{\line(1,0){0.47}}
\multiput(130.62,86.95)(0.46,-0.1){1}{\line(1,0){0.46}}
\multiput(131.09,86.85)(0.45,-0.11){1}{\line(1,0){0.45}}
\multiput(131.54,86.75)(0.45,-0.11){1}{\line(1,0){0.45}}
\multiput(131.99,86.64)(0.44,-0.12){1}{\line(1,0){0.44}}
\multiput(132.43,86.52)(0.43,-0.13){1}{\line(1,0){0.43}}
\multiput(132.86,86.39)(0.42,-0.13){1}{\line(1,0){0.42}}
\multiput(133.28,86.25)(0.41,-0.14){1}{\line(1,0){0.41}}
\multiput(133.69,86.11)(0.4,-0.15){1}{\line(1,0){0.4}}
\multiput(134.1,85.96)(0.39,-0.16){1}{\line(1,0){0.39}}
\multiput(134.49,85.81)(0.38,-0.16){1}{\line(1,0){0.38}}
\multiput(134.87,85.65)(0.37,-0.17){1}{\line(1,0){0.37}}
\multiput(135.25,85.48)(0.36,-0.17){1}{\line(1,0){0.36}}
\multiput(135.61,85.3)(0.17,-0.09){2}{\line(1,0){0.17}}
\multiput(135.96,85.12)(0.17,-0.09){2}{\line(1,0){0.17}}
\multiput(136.29,84.94)(0.16,-0.1){2}{\line(1,0){0.16}}
\multiput(136.62,84.75)(0.16,-0.1){2}{\line(1,0){0.16}}
\multiput(136.93,84.55)(0.15,-0.1){2}{\line(1,0){0.15}}
\multiput(137.22,84.35)(0.14,-0.1){2}{\line(1,0){0.14}}
\multiput(137.51,84.14)(0.13,-0.11){2}{\line(1,0){0.13}}
\multiput(137.78,83.93)(0.13,-0.11){2}{\line(1,0){0.13}}
\multiput(138.03,83.71)(0.12,-0.11){2}{\line(1,0){0.12}}
\multiput(138.27,83.49)(0.11,-0.11){2}{\line(1,0){0.11}}
\multiput(138.5,83.27)(0.11,-0.11){2}{\line(0,-1){0.11}}
\multiput(138.71,83.04)(0.1,-0.12){2}{\line(0,-1){0.12}}
\multiput(138.91,82.81)(0.09,-0.12){2}{\line(0,-1){0.12}}
\multiput(139.09,82.58)(0.16,-0.24){1}{\line(0,-1){0.24}}
\multiput(139.25,82.34)(0.15,-0.24){1}{\line(0,-1){0.24}}
\multiput(139.4,82.1)(0.13,-0.24){1}{\line(0,-1){0.24}}
\multiput(139.53,81.86)(0.12,-0.24){1}{\line(0,-1){0.24}}
\multiput(139.65,81.62)(0.1,-0.25){1}{\line(0,-1){0.25}}
\multiput(139.75,81.37)(0.08,-0.25){1}{\line(0,-1){0.25}}
\multiput(139.83,81.12)(0.07,-0.25){1}{\line(0,-1){0.25}}
\multiput(139.9,80.88)(0.05,-0.25){1}{\line(0,-1){0.25}}
\multiput(139.95,80.63)(0.03,-0.25){1}{\line(0,-1){0.25}}
\multiput(139.98,80.38)(0.02,-0.25){1}{\line(0,-1){0.25}}

\linethickness{0.55mm}
\put(110,80){\line(1,0){30}}
\linethickness{0.55mm}
\multiput(123.75,88.75)(0.12,-0.12){10}{\line(1,0){0.12}}
\linethickness{0.55mm}
\multiput(123.75,81.25)(0.12,-0.12){10}{\line(1,0){0.12}}
\linethickness{0.55mm}
\multiput(123.75,73.75)(0.12,-0.12){10}{\line(1,0){0.12}}
\linethickness{0.55mm}
\multiput(123.75,86.25)(0.12,0.12){10}{\line(1,0){0.12}}
\linethickness{0.55mm}
\multiput(123.75,71.25)(0.12,0.12){10}{\line(1,0){0.12}}
\linethickness{0.55mm}
\multiput(123.75,78.75)(0.12,0.12){10}{\line(1,0){0.12}}
\linethickness{0.3mm}
\put(140,80){\circle*{4.06}}

\linethickness{0.3mm}
\put(110,80){\circle*{4.06}}

\linethickness{0.55mm}
\multiput(63.28,88.75)(0.12,-0.12){10}{\line(1,0){0.12}}
\linethickness{0.55mm}
\multiput(63.28,73.75)(0.12,-0.12){10}{\line(1,0){0.12}}
\linethickness{0.55mm}
\multiput(63.28,86.25)(0.12,0.12){10}{\line(1,0){0.12}}
\linethickness{0.55mm}
\multiput(63.28,71.25)(0.12,0.12){10}{\line(1,0){0.12}}
\linethickness{0.55mm}
\put(79.53,79.87){\line(0,1){0.25}}
\multiput(79.51,79.62)(0.02,0.25){1}{\line(0,1){0.25}}
\multiput(79.48,79.37)(0.03,0.25){1}{\line(0,1){0.25}}
\multiput(79.43,79.12)(0.05,0.25){1}{\line(0,1){0.25}}
\multiput(79.36,78.88)(0.07,0.25){1}{\line(0,1){0.25}}
\multiput(79.28,78.63)(0.08,0.25){1}{\line(0,1){0.25}}
\multiput(79.18,78.38)(0.1,0.25){1}{\line(0,1){0.25}}
\multiput(79.06,78.14)(0.12,0.24){1}{\line(0,1){0.24}}
\multiput(78.93,77.9)(0.13,0.24){1}{\line(0,1){0.24}}
\multiput(78.78,77.66)(0.15,0.24){1}{\line(0,1){0.24}}
\multiput(78.62,77.42)(0.16,0.24){1}{\line(0,1){0.24}}
\multiput(78.44,77.19)(0.09,0.12){2}{\line(0,1){0.12}}
\multiput(78.24,76.96)(0.1,0.12){2}{\line(0,1){0.12}}
\multiput(78.03,76.73)(0.11,0.11){2}{\line(0,1){0.11}}
\multiput(77.8,76.51)(0.11,0.11){2}{\line(1,0){0.11}}
\multiput(77.56,76.29)(0.12,0.11){2}{\line(1,0){0.12}}
\multiput(77.31,76.07)(0.13,0.11){2}{\line(1,0){0.13}}
\multiput(77.04,75.86)(0.13,0.11){2}{\line(1,0){0.13}}
\multiput(76.75,75.65)(0.14,0.1){2}{\line(1,0){0.14}}
\multiput(76.46,75.45)(0.15,0.1){2}{\line(1,0){0.15}}
\multiput(76.15,75.25)(0.16,0.1){2}{\line(1,0){0.16}}
\multiput(75.82,75.06)(0.16,0.1){2}{\line(1,0){0.16}}
\multiput(75.49,74.88)(0.17,0.09){2}{\line(1,0){0.17}}
\multiput(75.14,74.7)(0.17,0.09){2}{\line(1,0){0.17}}
\multiput(74.78,74.52)(0.36,0.17){1}{\line(1,0){0.36}}
\multiput(74.4,74.35)(0.37,0.17){1}{\line(1,0){0.37}}
\multiput(74.02,74.19)(0.38,0.16){1}{\line(1,0){0.38}}
\multiput(73.63,74.04)(0.39,0.16){1}{\line(1,0){0.39}}
\multiput(73.22,73.89)(0.4,0.15){1}{\line(1,0){0.4}}
\multiput(72.81,73.75)(0.41,0.14){1}{\line(1,0){0.41}}
\multiput(72.39,73.61)(0.42,0.13){1}{\line(1,0){0.42}}
\multiput(71.96,73.48)(0.43,0.13){1}{\line(1,0){0.43}}
\multiput(71.52,73.36)(0.44,0.12){1}{\line(1,0){0.44}}
\multiput(71.07,73.25)(0.45,0.11){1}{\line(1,0){0.45}}
\multiput(70.62,73.15)(0.45,0.11){1}{\line(1,0){0.45}}
\multiput(70.15,73.05)(0.46,0.1){1}{\line(1,0){0.46}}
\multiput(69.69,72.96)(0.47,0.09){1}{\line(1,0){0.47}}
\multiput(69.21,72.87)(0.47,0.08){1}{\line(1,0){0.47}}
\multiput(68.73,72.8)(0.48,0.07){1}{\line(1,0){0.48}}
\multiput(68.25,72.73)(0.48,0.07){1}{\line(1,0){0.48}}
\multiput(67.76,72.68)(0.49,0.06){1}{\line(1,0){0.49}}
\multiput(67.27,72.63)(0.49,0.05){1}{\line(1,0){0.49}}
\multiput(66.78,72.58)(0.49,0.04){1}{\line(1,0){0.49}}
\multiput(66.28,72.55)(0.5,0.03){1}{\line(1,0){0.5}}
\multiput(65.78,72.53)(0.5,0.03){1}{\line(1,0){0.5}}
\multiput(65.28,72.51)(0.5,0.02){1}{\line(1,0){0.5}}
\multiput(64.78,72.5)(0.5,0.01){1}{\line(1,0){0.5}}
\put(64.28,72.5){\line(1,0){0.5}}
\multiput(63.78,72.51)(0.5,-0.01){1}{\line(1,0){0.5}}
\multiput(63.28,72.53)(0.5,-0.02){1}{\line(1,0){0.5}}
\multiput(62.78,72.55)(0.5,-0.03){1}{\line(1,0){0.5}}
\multiput(62.28,72.58)(0.5,-0.03){1}{\line(1,0){0.5}}
\multiput(61.79,72.63)(0.49,-0.04){1}{\line(1,0){0.49}}
\multiput(61.3,72.68)(0.49,-0.05){1}{\line(1,0){0.49}}
\multiput(60.81,72.73)(0.49,-0.06){1}{\line(1,0){0.49}}
\multiput(60.33,72.8)(0.48,-0.07){1}{\line(1,0){0.48}}
\multiput(59.85,72.87)(0.48,-0.07){1}{\line(1,0){0.48}}
\multiput(59.37,72.96)(0.47,-0.08){1}{\line(1,0){0.47}}
\multiput(58.91,73.05)(0.47,-0.09){1}{\line(1,0){0.47}}
\multiput(58.44,73.15)(0.46,-0.1){1}{\line(1,0){0.46}}
\multiput(57.99,73.25)(0.45,-0.11){1}{\line(1,0){0.45}}
\multiput(57.54,73.36)(0.45,-0.11){1}{\line(1,0){0.45}}
\multiput(57.1,73.48)(0.44,-0.12){1}{\line(1,0){0.44}}
\multiput(56.67,73.61)(0.43,-0.13){1}{\line(1,0){0.43}}
\multiput(56.25,73.75)(0.42,-0.13){1}{\line(1,0){0.42}}
\multiput(55.84,73.89)(0.41,-0.14){1}{\line(1,0){0.41}}
\multiput(55.43,74.04)(0.4,-0.15){1}{\line(1,0){0.4}}
\multiput(55.04,74.19)(0.39,-0.16){1}{\line(1,0){0.39}}
\multiput(54.66,74.35)(0.38,-0.16){1}{\line(1,0){0.38}}
\multiput(54.28,74.52)(0.37,-0.17){1}{\line(1,0){0.37}}
\multiput(53.92,74.7)(0.36,-0.17){1}{\line(1,0){0.36}}
\multiput(53.57,74.88)(0.17,-0.09){2}{\line(1,0){0.17}}
\multiput(53.24,75.06)(0.17,-0.09){2}{\line(1,0){0.17}}
\multiput(52.91,75.25)(0.16,-0.1){2}{\line(1,0){0.16}}
\multiput(52.6,75.45)(0.16,-0.1){2}{\line(1,0){0.16}}
\multiput(52.31,75.65)(0.15,-0.1){2}{\line(1,0){0.15}}
\multiput(52.02,75.86)(0.14,-0.1){2}{\line(1,0){0.14}}
\multiput(51.75,76.07)(0.13,-0.11){2}{\line(1,0){0.13}}
\multiput(51.5,76.29)(0.13,-0.11){2}{\line(1,0){0.13}}
\multiput(51.26,76.51)(0.12,-0.11){2}{\line(1,0){0.12}}
\multiput(51.03,76.73)(0.11,-0.11){2}{\line(1,0){0.11}}
\multiput(50.82,76.96)(0.11,-0.11){2}{\line(0,-1){0.11}}
\multiput(50.62,77.19)(0.1,-0.12){2}{\line(0,-1){0.12}}
\multiput(50.44,77.42)(0.09,-0.12){2}{\line(0,-1){0.12}}
\multiput(50.28,77.66)(0.16,-0.24){1}{\line(0,-1){0.24}}
\multiput(50.13,77.9)(0.15,-0.24){1}{\line(0,-1){0.24}}
\multiput(50,78.14)(0.13,-0.24){1}{\line(0,-1){0.24}}
\multiput(49.88,78.38)(0.12,-0.24){1}{\line(0,-1){0.24}}
\multiput(49.78,78.63)(0.1,-0.25){1}{\line(0,-1){0.25}}
\multiput(49.7,78.88)(0.08,-0.25){1}{\line(0,-1){0.25}}
\multiput(49.63,79.12)(0.07,-0.25){1}{\line(0,-1){0.25}}
\multiput(49.58,79.37)(0.05,-0.25){1}{\line(0,-1){0.25}}
\multiput(49.55,79.62)(0.03,-0.25){1}{\line(0,-1){0.25}}
\multiput(49.53,79.87)(0.02,-0.25){1}{\line(0,-1){0.25}}
\put(49.53,79.87){\line(0,1){0.25}}
\multiput(49.53,80.13)(0.02,0.25){1}{\line(0,1){0.25}}
\multiput(49.55,80.38)(0.03,0.25){1}{\line(0,1){0.25}}
\multiput(49.58,80.63)(0.05,0.25){1}{\line(0,1){0.25}}
\multiput(49.63,80.88)(0.07,0.25){1}{\line(0,1){0.25}}
\multiput(49.7,81.12)(0.08,0.25){1}{\line(0,1){0.25}}
\multiput(49.78,81.37)(0.1,0.25){1}{\line(0,1){0.25}}
\multiput(49.88,81.62)(0.12,0.24){1}{\line(0,1){0.24}}
\multiput(50,81.86)(0.13,0.24){1}{\line(0,1){0.24}}
\multiput(50.13,82.1)(0.15,0.24){1}{\line(0,1){0.24}}
\multiput(50.28,82.34)(0.16,0.24){1}{\line(0,1){0.24}}
\multiput(50.44,82.58)(0.09,0.12){2}{\line(0,1){0.12}}
\multiput(50.62,82.81)(0.1,0.12){2}{\line(0,1){0.12}}
\multiput(50.82,83.04)(0.11,0.11){2}{\line(0,1){0.11}}
\multiput(51.03,83.27)(0.11,0.11){2}{\line(1,0){0.11}}
\multiput(51.26,83.49)(0.12,0.11){2}{\line(1,0){0.12}}
\multiput(51.5,83.71)(0.13,0.11){2}{\line(1,0){0.13}}
\multiput(51.75,83.93)(0.13,0.11){2}{\line(1,0){0.13}}
\multiput(52.02,84.14)(0.14,0.1){2}{\line(1,0){0.14}}
\multiput(52.31,84.35)(0.15,0.1){2}{\line(1,0){0.15}}
\multiput(52.6,84.55)(0.16,0.1){2}{\line(1,0){0.16}}
\multiput(52.91,84.75)(0.16,0.1){2}{\line(1,0){0.16}}
\multiput(53.24,84.94)(0.17,0.09){2}{\line(1,0){0.17}}
\multiput(53.57,85.12)(0.17,0.09){2}{\line(1,0){0.17}}
\multiput(53.92,85.3)(0.36,0.17){1}{\line(1,0){0.36}}
\multiput(54.28,85.48)(0.37,0.17){1}{\line(1,0){0.37}}
\multiput(54.66,85.65)(0.38,0.16){1}{\line(1,0){0.38}}
\multiput(55.04,85.81)(0.39,0.16){1}{\line(1,0){0.39}}
\multiput(55.43,85.96)(0.4,0.15){1}{\line(1,0){0.4}}
\multiput(55.84,86.11)(0.41,0.14){1}{\line(1,0){0.41}}
\multiput(56.25,86.25)(0.42,0.13){1}{\line(1,0){0.42}}
\multiput(56.67,86.39)(0.43,0.13){1}{\line(1,0){0.43}}
\multiput(57.1,86.52)(0.44,0.12){1}{\line(1,0){0.44}}
\multiput(57.54,86.64)(0.45,0.11){1}{\line(1,0){0.45}}
\multiput(57.99,86.75)(0.45,0.11){1}{\line(1,0){0.45}}
\multiput(58.44,86.85)(0.46,0.1){1}{\line(1,0){0.46}}
\multiput(58.91,86.95)(0.47,0.09){1}{\line(1,0){0.47}}
\multiput(59.37,87.04)(0.47,0.08){1}{\line(1,0){0.47}}
\multiput(59.85,87.13)(0.48,0.07){1}{\line(1,0){0.48}}
\multiput(60.33,87.2)(0.48,0.07){1}{\line(1,0){0.48}}
\multiput(60.81,87.27)(0.49,0.06){1}{\line(1,0){0.49}}
\multiput(61.3,87.32)(0.49,0.05){1}{\line(1,0){0.49}}
\multiput(61.79,87.37)(0.49,0.04){1}{\line(1,0){0.49}}
\multiput(62.28,87.42)(0.5,0.03){1}{\line(1,0){0.5}}
\multiput(62.78,87.45)(0.5,0.03){1}{\line(1,0){0.5}}
\multiput(63.28,87.47)(0.5,0.02){1}{\line(1,0){0.5}}
\multiput(63.78,87.49)(0.5,0.01){1}{\line(1,0){0.5}}
\put(64.28,87.5){\line(1,0){0.5}}
\multiput(64.78,87.5)(0.5,-0.01){1}{\line(1,0){0.5}}
\multiput(65.28,87.49)(0.5,-0.02){1}{\line(1,0){0.5}}
\multiput(65.78,87.47)(0.5,-0.03){1}{\line(1,0){0.5}}
\multiput(66.28,87.45)(0.5,-0.03){1}{\line(1,0){0.5}}
\multiput(66.78,87.42)(0.49,-0.04){1}{\line(1,0){0.49}}
\multiput(67.27,87.37)(0.49,-0.05){1}{\line(1,0){0.49}}
\multiput(67.76,87.32)(0.49,-0.06){1}{\line(1,0){0.49}}
\multiput(68.25,87.27)(0.48,-0.07){1}{\line(1,0){0.48}}
\multiput(68.73,87.2)(0.48,-0.07){1}{\line(1,0){0.48}}
\multiput(69.21,87.13)(0.47,-0.08){1}{\line(1,0){0.47}}
\multiput(69.69,87.04)(0.47,-0.09){1}{\line(1,0){0.47}}
\multiput(70.15,86.95)(0.46,-0.1){1}{\line(1,0){0.46}}
\multiput(70.62,86.85)(0.45,-0.11){1}{\line(1,0){0.45}}
\multiput(71.07,86.75)(0.45,-0.11){1}{\line(1,0){0.45}}
\multiput(71.52,86.64)(0.44,-0.12){1}{\line(1,0){0.44}}
\multiput(71.96,86.52)(0.43,-0.13){1}{\line(1,0){0.43}}
\multiput(72.39,86.39)(0.42,-0.13){1}{\line(1,0){0.42}}
\multiput(72.81,86.25)(0.41,-0.14){1}{\line(1,0){0.41}}
\multiput(73.22,86.11)(0.4,-0.15){1}{\line(1,0){0.4}}
\multiput(73.63,85.96)(0.39,-0.16){1}{\line(1,0){0.39}}
\multiput(74.02,85.81)(0.38,-0.16){1}{\line(1,0){0.38}}
\multiput(74.4,85.65)(0.37,-0.17){1}{\line(1,0){0.37}}
\multiput(74.78,85.48)(0.36,-0.17){1}{\line(1,0){0.36}}
\multiput(75.14,85.3)(0.17,-0.09){2}{\line(1,0){0.17}}
\multiput(75.49,85.12)(0.17,-0.09){2}{\line(1,0){0.17}}
\multiput(75.82,84.94)(0.16,-0.1){2}{\line(1,0){0.16}}
\multiput(76.15,84.75)(0.16,-0.1){2}{\line(1,0){0.16}}
\multiput(76.46,84.55)(0.15,-0.1){2}{\line(1,0){0.15}}
\multiput(76.75,84.35)(0.14,-0.1){2}{\line(1,0){0.14}}
\multiput(77.04,84.14)(0.13,-0.11){2}{\line(1,0){0.13}}
\multiput(77.31,83.93)(0.13,-0.11){2}{\line(1,0){0.13}}
\multiput(77.56,83.71)(0.12,-0.11){2}{\line(1,0){0.12}}
\multiput(77.8,83.49)(0.11,-0.11){2}{\line(1,0){0.11}}
\multiput(78.03,83.27)(0.11,-0.11){2}{\line(0,-1){0.11}}
\multiput(78.24,83.04)(0.1,-0.12){2}{\line(0,-1){0.12}}
\multiput(78.44,82.81)(0.09,-0.12){2}{\line(0,-1){0.12}}
\multiput(78.62,82.58)(0.16,-0.24){1}{\line(0,-1){0.24}}
\multiput(78.78,82.34)(0.15,-0.24){1}{\line(0,-1){0.24}}
\multiput(78.93,82.1)(0.13,-0.24){1}{\line(0,-1){0.24}}
\multiput(79.06,81.86)(0.12,-0.24){1}{\line(0,-1){0.24}}
\multiput(79.18,81.62)(0.1,-0.25){1}{\line(0,-1){0.25}}
\multiput(79.28,81.37)(0.08,-0.25){1}{\line(0,-1){0.25}}
\multiput(79.36,81.12)(0.07,-0.25){1}{\line(0,-1){0.25}}
\multiput(79.43,80.88)(0.05,-0.25){1}{\line(0,-1){0.25}}
\multiput(79.48,80.63)(0.03,-0.25){1}{\line(0,-1){0.25}}
\multiput(79.51,80.38)(0.02,-0.25){1}{\line(0,-1){0.25}}

\linethickness{0.3mm}
\put(49.53,80){\circle*{4.06}}

\linethickness{0.3mm}
\put(79.53,80){\circle*{4.06}}

\linethickness{0.55mm}
\multiput(0,70)(0.12,0.24){83}{\line(0,1){0.24}}
\linethickness{0.55mm}
\multiput(10,90)(0.12,-0.24){83}{\line(0,-1){0.24}}
\linethickness{0.55mm}
\put(0,70){\line(1,0){20}}
\linethickness{0.55mm}
\multiput(8.75,71.25)(0.12,-0.12){10}{\line(1,0){0.12}}
\linethickness{0.55mm}
\multiput(8.75,68.75)(0.12,0.12){10}{\line(1,0){0.12}}
\linethickness{0.55mm}
\multiput(15,80)(0.38,-0.12){5}{\line(1,0){0.38}}
\linethickness{0.55mm}
\multiput(14.38,78.12)(0.12,0.38){5}{\line(0,1){0.38}}
\linethickness{0.55mm}
\multiput(5,80)(0.12,-0.38){5}{\line(0,-1){0.38}}
\linethickness{0.55mm}
\multiput(3.12,79.38)(0.38,0.12){5}{\line(1,0){0.38}}
\linethickness{0.3mm}
\put(20,70){\circle*{4.06}}

\linethickness{0.3mm}
\put(10,90){\circle*{4.06}}

\linethickness{0.3mm}
\put(0,70){\circle*{4.06}}

\end{picture}